\documentclass[12pt]{amsart}

\usepackage{amsmath,amssymb,amsthm}
\usepackage{amsmath,amssymb,amsthm,amscd}
\usepackage{enumerate}
\usepackage{graphpap, color}
\usepackage[mathscr]{eucal}
\usepackage{color}
\numberwithin{equation}{section}

\newtheorem{theo}{Theorem}[section]
\newtheorem{quest}{Question}[section]

\newtheorem{coro}{Corollary}[section]

\theoremstyle{definition}

\theoremstyle{remark}
\newtheorem{remark}{Remark}[section]

\newcommand{\p}{\partial}

\def\mby{\mathfrak{m}_{_{BY}}}
\def\be{\begin{equation}}
\def\ee{\end{equation}}

\begin{document}

\title[Uniqueness of isometric immersions]{Uniqueness of isometric immersions with the same mean curvature}

\author{Chunhe Li}
\address[Chunhe Li]{School of Mathematical Sciences  \\ University of Electronic Science and Technology of China \\ Chengdu, China} \email{chli@fudan.edu.cn}

\author{Pengzi Miao}
\address[Pengzi Miao]{Department of Mathematics, University of Miami, Coral Gables, FL 33146, USA.}
\email{pengzim@math.miami.edu}
\thanks{Research of the second author is partially supported by Simons Foundation Collaboration Grant for Mathematicians \#281105.}

\author{Zhizhang Wang}
\address[Zhizhang Wang]{School of Mathematical Sciences\\ Fudan University \\ Shanghai, China}
\email{zzwang@fudan.edu.cn}
\thanks{Research of the third author is partially supported  by NSFC Grants No.11301087, No.11671090 and No.11771103.}

\begin{abstract}
Motivated  by the quasi-local mass problem in general relativity,  
we study the rigidity of isometric  immersions  with the same mean curvature  into  
a warped product space.  
As a corollary of our main result, two  star-shaped hypersurfaces 
in a spatial Schwarzschild or AdS-Schwarzschild  manifold with nonzero mass 
 differ only by a rotation  if they are isometric and have the same  mean curvature.
We also prove  similar results if the mean curvature condition  is replaced by  
an $\sigma_2$-curvature condition. 
\end{abstract}


\renewcommand{\subjclassname}{
  \textup{2010} Mathematics Subject Classification}

\maketitle

\section{Introduction and statement of the results} \label{sec-intro}
In the quasi-local mass problem in general relativity, a basic pair  of geometric data $(g, H)$  
associated to a $2$-sphere  $\Sigma$, bounding some spacelike hypersurface 
$\Omega$ in a spacetime, consists of a Riemannian metric $g$ and a function $H$. Here $ g$ denotes the induced metric on $ \Sigma$ from $\Omega$ and $H$ is the mean curvature 
of $ \Sigma $ in $ \Omega$.  
For instance, the Brown-York quasi-local mass \cite{BY1, BY2} is given  by 
$$
\mby (\Sigma) = \frac{1}{8\pi} \left( \int_{\Sigma} H_0 - \int_\Sigma H  \right) ,
$$
where  the metric $g$  is assumed to have positive Gauss curvature
and $H_0$ is the mean curvature of the unique, isometric embedding of $ (\Sigma, g)$ into the $3$-dimensional 
Euclidean space $ \mathbb{R}^3$.

Existence and uniqueness of the 
 isometric embedding of $(\Sigma, g)$ into $\mathbb{R}^3$,
 used in defining $\mby (\Sigma)$,
 was guaranteed by Nirenberg's solution to the Weyl embedding problem (cf. \cite{Lewy, Nirenberg, Pogorelov, Weyl}).
 When the target space $ \mathbb{R}^3 $ is replaced by 
 the Minkowski spacetime $ \mathbb{R}^{3,1}$, 
 existence of isometric embeddings of  $(\Sigma, g)$ into $ \mathbb{R}^{3,1}$ was  
given  by Wang and Yau \cite{WY1, WY2}
 in association with the definition of Wang-Yau quasi-local mass.
If the target space is a general  Riemannian $3$-manifold  $N$, 
 estimates concerning the existence  of isometric  embeddings  into $N$ 
have been  studied  extensively in the literature. We refer readers to 
 the work in \cite{CX, Guan-Lu, Li-Wang, Lu, LW, Pogorelov2, PC} and the references therein.

In contrast to the embedding into $ \mathbb{R}^3$,  counterexamples were constructed 
in \cite{Li-Wang} to illustrate the lack of rigidity for
convex surfaces if the target  Riemannian manifold is not a space form. 
Such non-uniqueness, however, does not diminish the role played by isometric embeddings  
  in the  study of  relativistic problems. In  \cite{Lu-Miao},  
isometric embeddings of  $(\Sigma, g)$ into a spatial Schwarzschild $3$-manifold 
$$ (N_m, d s_m^2 ) = \left( [ 2m , \infty) \times \mathbb{S}^2, \frac{1}{1 - \frac{2 m}{r} } d r^2 + r^2 d \sigma  \right) $$
were used in deriving   a localized Riemannian Penrose inequality \cite{Bray, HI01}  that has a   form of  
$$
m + \frac{1}{8 \pi}  \int_\Sigma ( H_m - H) f \ge \sqrt{ \frac{ A}{ 16 \pi} } .
$$
Here $H_m$ is the mean curvature of an  embedding of $(\Sigma, g)$ in $(N_m, d s_m^2)$, 
$f $ is the static potential on $(N_m, d s^2)$ given by $ f  = \left( 1 - \frac{2m}{r} \right)^\frac12 $, and
$A$ represents the area of the horizon of black hole enclosed by $(\Sigma, g)$ in a 
physical manifold $\Omega$. 

The Schwarzschild manifold $(N_m, d s_m^2)$ is an example of a {\sl static} space 
$(N, d s^2)$, which by definition 
is a Riemannian $3$-manifold on which there exists a 
function $f>0$, referred as a static potential, such that the corresponding static spacetime
$$
(\bar N, \bar g ) = ( \mathbb{R}^1 \times N, - f^2 d t^2 + d s^2 ) 
$$
is Einstein. In \cite{CWWY}, isometric embeddings of $(\Sigma, g)$ into such a static Einstein 
spacetime $(\bar N, \bar g)$ were used as references in defining an analogue of  the Wang-Yau mass.
When  the image of such an  embedding lies in a constant $t$-slice of $(\bar N, \bar g)$, i.e. in $(N, d s^2)$, the associated quasi-local energy in \cite{CWWY}
becomes  the integral 
 $$ \frac{1}{8 \pi} \int_\Sigma ( H_s - H) f , $$
 where $H_s$ is the mean curvature of the embedding of $(\Sigma, g)$ in $(N, d s^2)$.

Motivated by the non-uniqueness example of isometric embeddings in \cite{Li-Wang} and 
by the results  in  \cite{CWWY, Lu-Miao} that make use of embeddings into static spaces,  
we ask  the following question:

\begin{quest} \label{quest-main}
{In a static  Riemannian $3$-manifold $(N, d s^2)$, if two surfaces $ \Sigma_1$ 
and $ \Sigma_2$  are isometric and have the same mean curvature (under the surface isometry), 
do $\Sigma_1$ and $ \Sigma_2$  differ only  by a rigid motion of the ambient space?}
\end{quest}

We recall that the static condition on $(N, d s^2)$ can be formulated equivalently (cf. \cite{Corvino}) 
in Riemannian terms by 
\begin{equation} \label{eq-static-intro}
 (\tilde \Delta f) \tilde g   -  \tilde \nabla^2 f  + f \mathrm{Ric} ( \tilde g)  = 0 ,
\end{equation}
where  $ \tilde \Delta$, $ \tilde \nabla^2 $ denote  the Laplacian, the Hessian of $\tilde g = d s^2$, respectively, and $ \mathrm{Ric} (\tilde g)$ is the Ricci curvature of $ \tilde g  $.

In this paper, we study Question \ref{quest-main} by  considering  
embeddings of an $n$-sphere $\mathbb{S}^n$, $ n \ge 2 $,  into 
an $(n+1)$-dimensional  warped product space
\be  \label{eq-warped-intro}
(N, ds^2) = \left( I \times \mathbb{S}^n, \frac{1}{f^2(r)}dr^2+r^2d\sigma \right), 
\ee
where  $ I  \subset \mathbb{R}^+ $  is an interval, 
$ f $ is a positive function on $I$, and 
$ d \sigma$ denotes the standard metric  of constant sectional 
curvature $1$ on $ \mathbb{S}^n$. 
In this case, a conformal Killing vector filed  on $(N, d s^2)$ is 
$X=r f \frac{\p}{\p r}$.
We define new radial coordinates $\rho$ and $ u$ on $I$ by
\be \label{eq-rho-u}
\rho = \frac12 X \cdot X = \frac12 r^2 \ \ \mathrm{and} \ \  u = \int \frac{d \rho}{ f (\rho) } ,
\ee
where `` $\cdot$ " denotes the metric $ d s^2$. 
We also define 
\begin{equation} \label{eq-big-phi}
\Phi =\frac{1}{2\rho}\left(ff_{\rho}+\frac{1-f^2}{2\rho}\right) .
\end{equation}
Direct calculation shows that  $(N, d s^2)$ is static with $f$ satisfying 
\eqref{eq-static-intro} if and only if 
\begin{eqnarray}\label{static-intro}
f_{uu}+(n-1)\Phi f=0 . 
\end{eqnarray}
Here  $ f_u, f_{\rho},  f_{uu},  f_{\rho \rho}$ denote
the first, second derivatives  of $f$ with respect to $u$, $\rho$,  respectively.

Given an immersion $\iota : \mathbb{S}^n \to (N, d s^2)$, let 
$\nu$ be  a chosen  unit  vector field  normal  to the hypersurface
$ M = \iota (\mathbb{S}^n)$, 
the {\sl support function}  of  $M$ is defined by 
\be \label{eq-phi}
\varphi = X \cdot \nu . 
\ee
With these notations,
we can state the main result of this paper.

\begin{theo} \label{theo-intro}
Suppose 
$(N, ds^2) = \left( I \times \mathbb{S}^n, \frac{1}{f^2(r)}dr^2+r^2d\sigma \right) $ satisfies 
\be \label{eq-super-static}
f_{uu} + (n-1)\Phi f \le 0 , 
\ee
with either $ \Phi \ge 0 $ or $ \Phi \le 0 $.
 Let $ g $ be a Riemannian metric on $\mathbb{S}^n$. Suppose 
 $ \iota$ and $ \tilde \iota $ are  two isometric immersions of 
 $ (\mathbb{S}^n, g) $ into $ (N, d s^2)$ such that 
 $ H = \tilde H$,  where $H$, $ \tilde H$ are the mean curvatures 
 of  the immersed surfaces $M = \iota (\mathbb{S}^n)$, $\tilde M = \tilde \iota (\mathbb{S}^n)$,
 respectively.  If $ M$ and $ \tilde M$  have positive support functions, then 
$$ h_{ij}  =  \tilde h_{ij} ,$$
where $ h_{ij}$, $ \tilde{h}_{ij}$ are the second fundamental forms of $M$, $ \tilde M$,
respectively. 

Moreover, if $ \Phi $ is  strictly positive or negative, 
then $M$ and $ \tilde M$ only differ by a rotation 
in $(N, d s^2)$. 
\end{theo}

\begin{remark} The assumption that $\Phi$ is positive or negative 
agrees with  the assumption (H4) or (H4') in Brendle's work on CMC surfaces in warped product spaces \cite{B2}. On the other hand, the 
assumption \eqref{eq-super-static}  appears to 
be opposite to  the assumption (H3) or (H3') in \cite{B2}.  
\end{remark}

\begin{remark} \label{rem-static}
When  equality in \eqref{eq-super-static} holds, i.e. 
when $(N, d s^2)$  is  static, it is known that 
$ f(r)$ is explicitly given  by 
$$ f = \left( 1 - {2 m}{r^{1-n}} + \kappa r^2   \right)^\frac12 $$
for some constants $m$ and $ \kappa$.
In this case,  $(N, d s^2)$ includes the spatial Schwarzschild 
and AdS-Schwarzschild manifolds and 
$ \Phi = (n+1) m r^{-n-3} $.
\end{remark}

As a direct corollary of Theorem \ref{theo-intro} and Remark \ref{rem-static}, we have 

\begin{coro} \label{cor-Sch}
Let $(N, ds^2)$ be a spatial Schwarzschild or AdS-Schwarzschild manifold with nonzero mass.
If $ M $ and $ \tilde M  $ are  two star-shaped hypersurfaces in $N$ such that 
 $M$ and $\tilde M$ are isometric and have the same mean curvature,  
 then $M$ and $ \tilde M$ only differ by a rotation in $(N, d s^2)$.  
\end{coro}

\begin{remark}
The spatial Schwarzschild manifold $(N_m, d s_m^2)$ is an example of an asymptotically flat manifold
(cf. \cite{Bray, HI01, SY79}).
If the ambient manifold $(N, ds^2)$ in Question \ref{quest-main}
is asymptotically flat and static, we suspect its answer is positive under suitable conditions on the surface. 
This is tied to the uniqueness aspect of the static metric extension conjecture  
formulated for the Bartnik quasi-local mass \cite{Bartnik}.
\end{remark}

When the target space is a space form, 
we note that the study of  isometrically immersed surfaces with the same  mean curvature 
has a longstanding history. 
A surface $S$ is called locally $H$-deformable if any point $x \in S$ has a neighborhood $U$ that 
admits a nontrivial $1$-parameter family of isometric deformations which preserve the mean curvature.     
Results on locally $H$-deformable surfaces can be found in \cite{Car, Chern, CK, Ken, La, Sch, Svec, Tr,  Um} and references therein.
Regarding global rigidity,  Lawson and Tribuzy \cite{La-Tr} proved that, for any 
compact oriented surface $\Sigma$ 
equipped with a Riemannian metric $g$ and given a non-constant function $H$ on $\Sigma$,  
there exist  at most two geometrically distinct isometric immersions of $(\Sigma, g)$ into a space form, 
with the  mean curvature function $H$.

The uniqueness  of  surfaces with prescribed extrinsic curvature alone  is also a classic problem in differential geometry. 
For instance,  the rigidity of  constant mean curvature hypersurfaces is an example of the prescribed mean curvature problem.  
A  theorem due to Alexandrov \cite{Alex} asserts that any closed, embedded hypersurface in $\mathbb{R}^n$ with constant mean curvature is a round sphere. Montiel \cite{Mon}  proved a uniqueness theorem for star-shaped hypersurfaces of constant mean curvature in certain rotationally symmetric manifolds. In \cite{B2}, Brendle obtained a generalization of Alexandrov's theorem for a class of warped product manifolds. 
In general, if one poses some function on the exterior unit normal vector field of a convex hypersurface in the Euclidean space, the uniqueness is still open except  for the $2$-dimensional case (cf.  \cite{GuWZ}). 
 
Prompted by Theorem \ref{theo-intro}, we also consider the rigidity question 
for  isometrically  immersed hypersurfaces with the same $ \sigma_2$-curvature (see definition \eqref{eq-sigma-2}).
We have the following result.

\begin{theo} \label{theo-intro-sigma-2}
Let $(N, ds^2) = \left( I \times \mathbb{S}^n, \frac{1}{f^2(r)}dr^2+r^2d\sigma \right)$
be a warped product space  with $\Phi\Phi_u>0$. 
Let $g$ be a Riemannian metric on   $ \mathbb{S}^n$. 
Suppose  $(\mathbb{S}^n, g)$ can be isometrically  immersed into $(N, d s^2)$ as 
two hypersurfaces $M$ and $\tilde M$. If $M$ and $ \tilde M$ have the same $\sigma_2$-curvature, 
then they differ only by a rotation in $(N, d s^2)$.
 \end{theo}

The remainder of this paper is organized as follows.
In Section \ref{sec-prelim}, we collect some basic formulae on 
immersed hypersurfaces in a warped product space.
In Section \ref{sec-inf-rigi}, we prove 
the infinitesimal rigidity of isometric surfaces with the same mean  
curvature via integral identities. 
The method is then revised in Section \ref{sec-global-rigi}  
to derive the global rigidity, hence proving Theorem \ref{theo-intro}. 
In Section \ref{sec-other}, 
we consider the analogue of Theorem \ref{theo-intro} with the mean curvature condition replaced by the $\sigma_2$-curvature condition and prove Theorem \ref{theo-intro-sigma-2}.
As an additional application of the method used in Sections \ref{sec-inf-rigi} and \ref{sec-global-rigi},
we also give  another proof of the infinitesimal rigidity and global rigidity of convex surfaces in space forms 
(see Theorem \ref{theo-rigidity-space-form}).

\vspace{.2cm}

We recently have learned  that Po-Ning Chen and Xiangwen Zhang \cite{Chen-Zhang}
proved a  rigidity result for surfaces in $3$-dimensional spatial Schwarzschild manifold which is 
similar but different to that in Corollary \ref{cor-Sch}.
We want to thank the authors for  the communication on their paper.  

 \section{Preliminaries on hypersurfaces in a warped product space} \label{sec-prelim}

 Let $ \Sigma = \mathbb{S}^n$ and let $ g$ be a Riemannian metric on $\Sigma$.
 Suppose $ \iota : (\Sigma, g) \rightarrow (N, d s^2)$ is an isometric 
 immersion, where $(N,ds^2)$ is  given in  \eqref{eq-warped-intro}.  
 Let $ M = \iota (\Sigma)$ and  let $D$ denote the Levi-Civita connection on $(N, ds^2)$.
 Let  $ \rho$, $ u$, $ \Phi$, $ \nu$ and $ \varphi$ be given in \eqref{eq-rho-u} -- \eqref{eq-phi}.
 
 Suppose  $\{e_1,e_2,\cdots, e_n\}$ is a  local  frame on $(\Sigma, g)$, 
 which can be viewed as a frame on $M$ via $\iota$.
 Let  $h_{ij}$ denote the second fundamental form of $M$, defined by 
 $$D_{e_i}e_j=-h_{ij}\nu.$$ The following formulae can be easily checked (cf. \cite{GLW, Li-Wang}):
\begin{eqnarray}
\varphi^2   & = &  2\rho-\frac{ | \nabla\rho|^2}{f^2} = 2\rho-|\nabla u|^2 ,  \label{0.1} \\
 h_{ij}\varphi&=& -\frac{\rho_{i,j}}{f}+\frac{f_{\rho}}{f^2}\rho_i\rho_j+fg_{ij}=-u_{i,j}+fg_{ij},  \label{h-phi} \\
 \varphi_i & = & \sum_kh_{ik}X\cdot e_k=\sum_{k}h_{ik}\frac{\rho_k}{f}=\sum_kh_{ik}u_k. \label{var}
\end{eqnarray}
Here a function with a lower index $i$ denotes its derivative along  $e_i$ and 
 `` , " denotes the covariant differentiation on $M$ or equivalently on $(\Sigma, g)$.

Given vector fields $Y$, $Z$, $W$ on $N$, 
let the curvature tensor on $(N, d s^2)$ be  given by 
$$\tilde{R}(Y, Z) W =D_Y D_Z W  - D_Z D_Y  W  -D_{[Y,Z]} W .$$
 Direct calculation (see (2.7) in  \cite{Li-Wang}) shows 
\begin{equation} \label{eq-sec-T}
 \begin{split}
 (-1)  \sum_{i < j } \tilde{R}(e_i, e_j) e_j \cdot e_i   
 & =   (n-1)\left[ff_{\rho}+\frac{n-2}{4}\frac{f^2-1}{\rho}-\varphi^2\frac{2\rho ff_{\rho}+1-f^2}{4\rho^2}\right]\\
 &=   (n-1)\left(f_u +\frac{n-2}{4}\frac{f^2-1}{\rho}- \varphi^2\Phi\right) \\
 &=  (n-1)\left[ \frac{n}{2}f_u -(n-2)\rho\Phi- \varphi^2\Phi\right]. 
  \end{split}
 \end{equation}
Let $ R$ be the scalar curvature  of $(\Sigma, g)$.
By the Gauss equations, we have 
\begin{equation}\label{2.5}
\sigma_2(h) =\frac{R}{2}+ (n-1)\left[ \frac{n}{2}f_u -(n-2)\rho\Phi- \varphi^2\Phi\right]  . 
\end{equation}
Here $\sigma_2 (h) = \sum_{i<j} \kappa_i  \kappa_j$, 
where $\{ \kappa_i \}_{i=1}^n$ are the principal curvature of $M$.
The Codazzi equations are given by 
\begin{eqnarray}
h_{ij,k}-h_{ik,j}&=&   \tilde{R} ( e_j, e_k) e_i \cdot \nu,
\end{eqnarray}
which implies 
\begin{eqnarray}\label{3.6A}
\sum_ih_{ij , i}-H_j&=&  \tilde{Ric}_{j\nu}.
\end{eqnarray} 
Here $\tilde {Ric} = \mathrm{Ric} (\tilde g)$ denotes  the Ricci curvature of $(N, d s^2)$.
The following formula can be checked:
\be \label{eq-Ric-jnu}
\tilde {Ric}_{j\nu} =  - ( n - 1 )\varphi \Phi u_j. 
\ee
(For instance, \eqref{eq-Ric-jnu} follows from (2.3) in \cite{Li-Wang}.)
Thus,  \eqref{3.6A} becomes
\begin{eqnarray}\label{2.6}
\sum_ih_{ij,i}-H_j&=&-(n-1)\varphi\Phi u_j.
\end{eqnarray} 

The geometric meaning of  $ \Phi$ is as follows:
\be
\tilde {Ric} ( E_1, E_1) - \tilde {Ric} (V, V)   = - (n-1) 2 \rho \Phi ,
\ee
where  $ E_1 = f \p_r $ is the unit normal to $ S_r =  \{ r \} \times \mathbb{S}^n$
and $V$ denotes  any unit vector tangent to $S_r$. 
In relation to equation \eqref{eq-static-intro},
we also note 
\be
\left[ (\tilde \Delta f) \tilde g   -  \tilde \nabla^2 f  + f \tilde {Ric} \right] (V, V) = 
2 \rho \left[ f_{uu} + (n-1) \Phi f \right].
\ee

By \eqref{eq-big-phi}, $\Phi$ can be rewritten as 
$$ \Phi = \frac14 \left( \frac{ f^2 - 1}{\rho} \right)_\rho. $$
From this, it is easily seen  
\begin{equation}
\left(ff_{\rho} \right)_{\rho}-4\Phi-2\rho\Phi_{\rho}=0,\nonumber
\end{equation}
or  equivalently
\begin{equation}\label{2.10}
f_{uu}-4 \Phi f - 2\rho\Phi_{u}=0.
\end{equation}

\section{Infinitesimal rigidity for surfaces with fixed mean curvature} \label{sec-inf-rigi} 
In this section,  we prove  the infinitesimal rigidity for isometric surfaces with the same mean curvature.
Suppose $\{ \iota_t \}$ is a $1$-parameter family of isometric immersions of $(\Sigma, g)$ in $(N, d s^2)$ 
which have the same mean curvature function $H$. 
Let $ M_t = \iota_t (\Sigma)$ and $M = M_0$.
Let an upper  dot denote derivative with respect to $t$ at $ t =0 $. 
Then 
\begin{equation}\label{3.1}
\left.\dot{H}=\frac{dH}{dt}\right|_{t=0}=0;\ \ \dot{g}_{ij}=\left.\frac{d g_{ij}}{dt}\right|_{t=0}=0.
\end{equation}
The infinitesimal rigidity means that we want to show  $\dot{h}_{ij}=0$.

Let $\{e_i\}_{1\le i \le n}$ be  a local orthonormal frame  on $(\Sigma, g)$. 
By  \eqref{0.1} and \eqref{h-phi}, we have
\begin{eqnarray}\label{1.3}
\varphi\dot{\varphi}=f\dot{u}-\nabla u\cdot\nabla\dot{u},
\end{eqnarray}
\begin{eqnarray}\label{1.4}
\dot{h}_{ij}\varphi+h_{ij}\dot{\varphi}=-\dot{u}_{i,j}+f_{u}\dot{u}\delta_{ij}.
\end{eqnarray}
The linearization of \eqref{2.5} and \eqref{2.6}  implies 
\begin{equation}\label{0.16}
\begin{split}
& \ \sum_{i,j}h_{ij}\dot{h}_{ij}  = H\dot{H} - \dot{\sigma}_2(h) \\
=& \ -(n-1)\left(\frac{n}{2}f_{uu} \dot{u}-(n-2)f\Phi\dot{u}-(n-2)\rho\Phi_u\dot{u}-2\varphi \dot{\varphi}\Phi-\varphi^2\Phi_{u}\dot{u}\right), 
\end{split}
\end{equation}
\begin{eqnarray}\label{1.8}
\sum_i\dot{h}_{ij,i}&=&-(n-1)(\dot{u}_j\varphi \Phi+u_j \dot{\varphi}\Phi+u_j\varphi\Phi_{u}\dot{u}),
\end{eqnarray}
where we  have used $\dot{H}=0$. 

Let   $w $ be an auxiliary function defining  on $I$, which will be determined later.
Viewing $w$ as a function on $ N $ and pulling it back to $\Sigma$,  we have by \eqref{1.4}, 
\begin{equation}\label{3.12}
\dot{h}_{ij}\varphi w=-\dot{u}_{i,j}w+f_{u}\dot{u}w\delta_{ij}-h_{ij}w\dot{\varphi}.
\end{equation}
Integrating on $(\Sigma, g)$ and using $\dot H =0$, we  have 
\begin{equation*} 
\begin{split}
\int_\Sigma \sum_{ij}\dot{h}_{ij}^2\varphi w 
=&-\sum_{i,j}\int_\Sigma \dot{u}_{i,j}w\dot{h}_{ij}+\int_\Sigma \dot{u}w\dot{H}-\sum_{i,j}\int_\Sigma h_{ij}\dot{h}_{ij}w\dot{\varphi} \\
=&\sum_{i,j}\int_\Sigma  \dot{u}_{i}w_j\dot{h}_{ij}+\sum_{i,j}\int_\Sigma \dot{u}_{i}w\dot{h}_{ji,j}-\sum_{i,j}\int_\Sigma  h_{ij}\dot{h}_{ij}w\dot{\varphi} \\
=&-\sum_{i,j}\int_\Sigma  \dot{u}w_{i,j}\dot{h}_{ij}-\sum_{i,j}\int_\Sigma \dot{u}w_j\dot{h}_{ij,i}+\sum_{i,j}\int_\Sigma \dot{u}_{i}w\dot{h}_{ji,j}-\sum_{i,j}\int_\Sigma h_{ij}\dot{h}_{ij}w\dot{\varphi} \\
=&-\sum_{i,j}\int_\Sigma  \dot{u}w_{uu}u_iu_j\dot{h}_{ij}-\sum_{i,j}\int_\Sigma  \dot{u}w_j\dot{h}_{ij,i}+\sum_{i,j}\int_\Sigma  \dot{u}_{i}w\dot{h}_{ji,j}-\sum_{i,j}\int_\Sigma 
h_{ij}\dot{h}_{ij}w\dot{\varphi} \\
&+\sum_{i,j}\int_\Sigma w_u\dot{u}(-u_{i,j}+f\delta_{ij})\dot{h}_{ij}-\int_\Sigma w_u\dot{u}f\dot{H} \\
=&-\sum_{i,j}\int_\Sigma \dot{u}w_{uu}u_iu_j\dot{h}_{ij}+\sum_{i,j}\int_\Sigma (\dot{u}_{i}w-\dot{u}w_i)\dot{h}_{ji,j}-\sum_{i,j}\int_\Sigma h_{ij}\dot{h}_{ij}(w\dot{\varphi}-w_u\dot{u}\varphi) .
\end{split}
\end{equation*}
Applying   \eqref{0.16} and  \eqref{1.8}, we then  have
\begin{equation}\label{1.11}
\begin{split}
& \int_\Sigma \sum_{ij}\dot{h}_{ij}^2\varphi w  \\
=&-\sum_{i,j}\int_\Sigma w_{uu}\dot{u}u_iu_j\dot{h}_{ij} 
-(n-1)\sum_{i}\int_\Sigma(\dot{u}_{i}w-\dot{u}w_u u_i)(\dot{u}_i\varphi \Phi+u_i \dot{\varphi}\Phi+u_i\varphi\Phi_{u}\dot{u}) \\
&+(n-1)\int_\Sigma (w\dot{\varphi}-w_u\dot{u}\varphi)\left(\frac{n}{2}f_{uu} \dot{u}-(n-2)(f\Phi+\rho\Phi_u)\dot{u}-2\varphi \dot{\varphi}\Phi-\varphi^2\Phi_{u}\dot{u}\right) \\
=&-\sum_{i,j}\int_\Sigma w_{uu}\dot{u}u_iu_j\dot{h}_{ij}-(n-1)\int_\Sigma (|\nabla\dot{u}|^2+2\dot{\varphi}^2)w\varphi \Phi \\
&+(n-1)\int_\Sigma \dot{u}\dot{\varphi}\Big(|\nabla u|^2w_u\Phi+\frac{n}{2}wf_{uu}-(n-2)(f\Phi+\rho\Phi_u)w 
+2\varphi^2w_u\Phi-w\varphi^2\Phi_u\Big) \\
&-(n-1)\int_\Sigma \dot{u}^2\Big(\frac{n}{2}w_u\varphi f_{uu}-(n-2)(f\Phi+\rho\Phi_u)w_u\varphi 
-|\nabla u|^2w_u\varphi\Phi_u-w_u\varphi^3\Phi_u\Big) \\
&-(n-1)\int_\Sigma \nabla u\cdot\nabla \dot{u} (-\dot{u}w_u\varphi\Phi+\dot{\varphi}\Phi w+\dot{u}w\varphi\Phi_u) .
\end{split}
\end{equation}
To proceed, we note that 
\begin{eqnarray}\label{1.13}
&&\int_\Sigma |\nabla \dot{u}|^2w\varphi\Phi  =\int_\Sigma   w\varphi\Phi \nabla\dot{u}\cdot \nabla \dot{u}\\
&=&-\int_\Sigma  \dot{u}(\Delta \dot{u} )w\varphi\Phi-\int_\Sigma \dot{u}w\Phi\nabla\dot{u}\cdot \nabla \varphi-\int_\Sigma \dot{u}\varphi(w_u\Phi+w\Phi_u)\nabla\dot{u}\cdot\nabla u\nonumber,
\end{eqnarray}
and 
\begin{eqnarray}\label{1.14}
&&\int_\Sigma \dot{\varphi}^2\varphi w\Phi=\int_\Sigma (f\dot{u}-\nabla u\cdot\nabla \dot{u})\dot{\varphi} w\Phi\\
&=&\int_\Sigma \dot{u}\dot{\varphi} fw\Phi+\int_\Sigma \dot{u} \dot{\varphi}w\Phi\Delta u+\int_\Sigma \dot{u}w\Phi \nabla u\cdot \nabla \dot{\varphi}+\int_\Sigma \dot{u}\dot{\varphi}(w_u\Phi+w\Phi_u)|\nabla u|^2\nonumber.
\end{eqnarray}
Moreover, the linearization of \eqref{var} gives 
\begin{equation}
\dot{\varphi}_i=\sum_k\dot{h}_{ik}u_k+\sum_kh_{ik}\dot{u}_k.
\end{equation}
 Thus, we have 
\begin{eqnarray}\label{1.15}
\nabla u\cdot \nabla\dot{\varphi}-\nabla \dot{u}\cdot \nabla\varphi=\sum_{i,j} \dot{h}_{ij} u_iu_j.
\end{eqnarray}
It follows from \eqref{1.13} -- \eqref{1.15} that
\begin{eqnarray}\label{1.16}
&&\int_\Sigma (|\nabla\dot{u}|^2+\dot{\varphi}^2)w\varphi\Phi\\
&=&\int_\Sigma  w\Phi \dot{u}\dot{h}_{ij}u_iu_j+\int_\Sigma \dot{u} w\Phi(\dot{\varphi}\Delta u-\varphi\Delta\dot{u})-\int_\Sigma \dot{u}\varphi(w_u\Phi+w\Phi_u)\nabla u\cdot\nabla \dot{u}\nonumber\\
&&+\int_\Sigma \dot{u}\dot{\varphi}(fw\Phi+(w_u\Phi+w\Phi_u)|\nabla u|^2)\nonumber.
\end{eqnarray}
Using  \eqref{h-phi}  and \eqref{1.4}, we have
$$\dot{\varphi}\Delta u-\varphi\Delta\dot{u}=\dot{\varphi}(nf-H\varphi)-\varphi(nf_u\dot{u}-H\dot{\varphi})=n(f\dot{\varphi}-\varphi f_u\dot{u}),$$ where $\dot{H}=0$ has been used. Thus \eqref{1.16} becomes
\begin{eqnarray}\label{1.17}
&&\int_\Sigma (|\nabla\dot{u}|^2+\dot{\varphi}^2)w\varphi\Phi\\
&=&\int_\Sigma  w\Phi \dot{u}\dot{h}_{ij}u_iu_j-n\int_\Sigma \dot{u}^2 wf_u\Phi\varphi-\int_\Sigma \dot{u}\varphi(w_u\Phi+w\Phi_u)\nabla u\cdot\nabla \dot{u}\nonumber\\
&&+\int_\Sigma \dot{u}\dot{\varphi}((n+1)fw\Phi+(w_u\Phi+w\Phi_u)|\nabla u|^2)\nonumber.
\end{eqnarray}
Combing \eqref{1.11} with \eqref{1.17}, we obtain
\begin{eqnarray}\label{1.11-1}
&&\int_\Sigma  \sum_{ij}\dot{h}_{ij}^2\varphi w\\
&=&-\sum_{i,j}\int_\Sigma (w_{uu}+(n-1)\Phi w)\dot{u}u_iu_j\dot{h}_{ij}\nonumber\\
&&-(n-1)\int_\Sigma \nabla u\cdot\nabla \dot{u} (-2\dot{u}w_u\varphi\Phi+\dot{\varphi}\Phi w)-(n-1)\int_\Sigma \dot{\varphi}^2w\varphi \Phi\nonumber\\
&&-(n-1)\int_\Sigma \dot{u}^2\Big(\frac{n}{2}w_u\varphi f_{uu}-(n-2)(f\Phi+\rho\Phi_u)w_u\varphi\nonumber\\
&&\ \ \ \ \ \ \ \ \ \ \ \ \ \ \ \ \  -|\nabla u|^2w_u\varphi\Phi_u-w_u\varphi^3\Phi_u-nwf_u\varphi\Phi\Big)\nonumber\\
&&+(n-1)\int_\Sigma \dot{u}\dot{\varphi}\Big(-|\nabla u|^2w\Phi_u+\frac{n}{2}wf_{uu}-(n-2)(f\Phi+\rho\Phi_u)w\nonumber\\
&&\ \ \ \ \ \ \ \ \ \ \ \ \ \ \ \ \  +2\varphi^2w_u\Phi-w\varphi^2\Phi_u-(n+1)fw\Phi\Big)\nonumber.
\end{eqnarray}
Applying   \eqref{1.3} and \eqref{0.1} to replace the terms 
$\nabla u\cdot \nabla \dot{u}$ and $|\nabla u|^2$ in \eqref{1.11-1},   we have
\begin{eqnarray}\label{1.18}
&&\int_\Sigma  \sum_{ij}\dot{h}_{ij}^2\varphi w\\
&=&-\sum_{i,j}\int_\Sigma (w_{uu}+(n-1)\Phi w)\dot{u}u_iu_j\dot{h}_{ij}\nonumber\\
&&+\frac{n(n-1)}{2}\int_\Sigma \dot{u}\dot{\varphi}(-2\rho w\Phi_u+wf_{uu}-4fw\Phi)\nonumber\\
&&-\frac{n(n-1)}{2}\int_\Sigma \dot{u}^2(w_u\varphi f_{uu}-2\rho w_u\varphi\Phi_u-2wf_u\varphi\Phi-2fw_u\varphi\Phi)\nonumber\\
&=&-\sum_{i,j}\int_\Sigma (w_{uu}+(n-1)\Phi w)\dot{u}u_iu_j\dot{h}_{ij}-n(n-1)\int_\Sigma \dot{u}^2\varphi\Phi(w_u f-wf_u)\nonumber,
\end{eqnarray}
where \eqref{2.10} has been used in the second equality.
Thus, we obtain
\begin{equation}\label{1.19}
\begin{split}
0 = & \int_\Sigma  \sum_{ij}\dot{h}_{ij}^2\varphi w+\sum_{i,j}\int_\Sigma (w_{uu}+(n-1)\Phi w)\dot{u}u_iu_j\dot{h}_{ij}  \\
& +n(n-1)\int_\Sigma \dot{u}^2\varphi\Phi(w_u f-wf_u) .
\end{split}
\end{equation}

If $(N, d s^2)$ is static with $f$ satisfying \eqref{eq-static-intro}, 
we can choose $w = f$. 
In view of  \eqref{static-intro} and \eqref{1.19}, 
we then conclude  $\dot{h}_{ij}=0$ provided the  support function $\varphi$ is positive. 

Next suppose condition \eqref{eq-super-static} holds.
Let $ (r_0 , r_1 ) \subset  I $ be an interval  so that 
 the coordinate  $r$ varies between $r_0$ and $r_1$ on the surface $M$. 
We first consider the case $ \Phi \ge 0 $.
In this case, let $ w $ be a solution to the  linear second order ordinary differential equation 
\begin{eqnarray}\label{equ}
w_{uu}+(n-1)\Phi w=0
\end{eqnarray} 
with the initial conditions $w(r_0)>0 $, $ w_u(r_0)>0$
and 
\be \label{eq-initial-w}
w_u(r_0)f(r_0)-w(r_0)f_u(r_0) \ge 0 .
\ee
(Note that such a $ w$ always exists.) 
Then we have
\begin{eqnarray}\label{1.22}
\left(\frac{w}{f}\right)_{uu}+2\frac{f_u}{f}\left(\frac{w}{f}\right)_u+\frac{f_{uu}+(n-1)\Phi f}{f}\frac{w}{f}=0, 
\end{eqnarray}
where 
$$ 
\frac{f_{uu}+(n-1)\Phi f}{f} \le 0 
$$ 
by \eqref{eq-super-static}.
We claim that $ w/f $ is always positive when $ r $ ranges in $ [r_0 , r_1]$. 
To see this, 
suppose  $r_* \in (r_0, r_1]$ is the first zero of $w/f$ from $r_0$. As $w_u(r_0)>0$, 
$ \max_{r \in [r_0, r_*] } w / f $ must 
occur in $(r_0,r_*)$, which contradicts the strong maximum principle.  
Therefore, $ w > 0 $ for $ r \in (r_0, r_1)$.
Consequently, 
\begin{eqnarray}\label{increase}
(w_u f-wf_u)_u=w_{uu}f-wf_{uu}=-w(f_{uu}+(n-1)\Phi f)\geq 0,
\end{eqnarray}
by \eqref{eq-super-static}.
Hence,  by \eqref{eq-initial-w},  
$ w_u f-wf_u \ge  0 $.
This allows us to conclude $\dot{h}_{ij}=0$ from  \eqref{1.19} under the assumption 
$ \varphi > 0$.
If $ \Phi > 0 $ on $N$, by choosing the inequality in \eqref{eq-initial-w} to be 
strict, we have $ w_u f-wf_u > 0 $. In this case, it also follows from  \eqref{1.19}
that $ \dot {u} = 0 $.

If $ \Phi \le 0 $, similar arguments apply   if we choose $ w$ 
so that 
$w(r_1)>0$, $w_u(r_1)<0$ and 
\be \label{eq-initial-w-2}
w_u(r_1)f(r_1)-w(r_1)f_u(r_1) \le 0. 
\ee
When $ \Phi < 0 $, we choose the inequality in \eqref{eq-initial-w-2}  to be strict.

The above discussion  has proved  the following theorem. 

\begin{theo} \label{theo-inf-H}
Suppose $(N, ds^2) = \left( I \times \mathbb{S}^n, \frac{1}{f^2(r)}dr^2+r^2d\sigma \right)$
satisfies \eqref{eq-super-static},
with  either $ \Phi \ge 0 $ or $ \Phi \le 0 $.
Let $g$ be a Riemannian metric on  
$ \Sigma =  \mathbb{S}^n$.
If $(\Sigma, g)$ can be isometrically  immersed into $(N, d s^2)$ as a hypersurface $M$ 
with positive support function,
then the infinitesimal rigidity of $M$ with fixed mean curvature holds.
More precisely, this means that if $(\Sigma, g)$ admits a 
family of isometric immersion $\{ \iota_t \}$ such that $ M_t = \iota_t (\Sigma)$ 
has the same mean curvature and $M_0 = M$, 
then the  linearization of the second fundamental form of $M_t$ at $M=M_0$ is trivial.       
Moreover, if $ \Phi  > 0 $ or $ \Phi < 0 $, 
 the  linearization of the function $u$  at $M=M_0$ is also trivial.
\end{theo}

\section{Global rigidity for surfaces with fixed  mean curvature} \label{sec-global-rigi}
In this section, we will modify the proof in Section \ref{sec-inf-rigi} to prove the global rigidity, 
i.e. Theorem \ref{theo-intro}.
Suppose $ \iota$, $ \tilde \iota$ are two isometric immersions of $(\Sigma, g)$ in $(N, d s^2)$.
Let $M = \iota (\Sigma) $ and $\tilde M = \tilde \iota (\Sigma)$.
We denote the restriction of $X $, $ \rho$, $u$ (which is defined in the ambient space $N$)
to $\tilde M$ by $\tilde X$, $ \tilde \rho $, $\tilde u$,  respectively. 
By abuse of notations, we also use $X $, $ \rho$, $u$ to denote their restriction to $M$. 
Let $ \varphi$, $ h $ and $ \tilde \varphi$, $ \tilde h$ be the support function, the second fundamental
form of $M$ and  $\tilde M$, respectively. 
Viewing $h$ and $\tilde h$ as tensors on $\Sigma$, we define 
$$ v = \tilde h - h . $$

Now suppose  $M$ and $ \tilde M$ have the same mean curvature. 
Let  $\{e_1,e_2,\cdots,e_n\}$ be a local orthonormal frame on $(\Sigma, g)$, we have 
\begin{eqnarray}\label{3.1-g}
\sum_iv_{ii}=0.
\end{eqnarray}
By the Gauss equations of $M$ and $\tilde{M}$, the corresponding equations of \eqref{2.5} are
\begin{eqnarray}\label{3.2}
\sigma_2(\tilde{h})&=&\frac{\tilde{R}}{2}+(n-1)\left(\frac{n}{2}f_u(\tilde{u})-(n-2)\tilde{\rho}\Phi(\tilde{u})-\tilde{\varphi}^2\Phi(\tilde{u})\right), \\
\sigma_2(h)&=&\frac{R}{2}+(n-1)\left(\frac{n}{2}f_u(u)-(n-2)\rho\Phi(u)-\varphi^2\Phi(u)\right)\nonumber.
\end{eqnarray}
Since $\tilde{R}=R$, the difference of the above two equations gives 
\begin{eqnarray}
&&\sum_{i<j}\left(h_{ii}v_{jj}+h_{jj}v_{ii}-2h_{ij}v_{ij}+v_{ii}v_{jj}-v_{ij}^2\right)\\
&=&\frac{ n (n-1) }{2}(f_u(\tilde{u})-f_u(u))-(n-2)(n-1)(\tilde{\rho}\Phi(\tilde{u})-\rho\Phi(u))\nonumber\\
&&-(n-1)\left(\tilde{\varphi}^2\Phi(\tilde{u})-\varphi^2\Phi(u)\right) . \nonumber
\end{eqnarray}
 Using \eqref{3.1-g}, we can rewrite the above  as 
 \begin{equation} \label{3.3}
  \begin{split}
  -\sum_{i,j}h_{ij}v_{ij} =&\frac{(n-1)n}{2}(f_u(\tilde{u})-f_u(u))-(n-2)(n-1)(\tilde{\rho}\Phi(\tilde{u})-\rho\Phi(u))\\
  &-(n-1)(\tilde{\varphi}^2\Phi\left(\tilde{u})-\varphi^2\Phi(u)\right)+\frac{1}{2}\sum_{i,j}v_{ij}^2, \\
  -\sum_{i,j}\tilde{h}_{ij}v_{ij} =&\frac{(n-1)n}{2}(f_u(\tilde{u})-f_u(u))-(n-2)(n-1)(\tilde{\rho}\Phi(\tilde{u})-\rho\Phi(u)) \\
  &-(n-1)(\tilde{\varphi}^2\Phi\left(\tilde{u})-\varphi^2\Phi(u)\right)-\frac{1}{2}\sum_{i,j}v_{ij}^2. 
 \end{split}
  \end{equation}
By the Codazzi equations of $M$ and $ \tilde{M}$,  the corresponding equations of  \eqref{2.6} are
\begin{eqnarray}
&&\sum_ih_{ij, i}-H_{j}=-(n-1)\varphi\Phi(u)u_j,\nonumber\\
&&\sum_i\tilde{h}_{ij, i}-\tilde{H}_{j}=-(n-1)\tilde{\varphi}\Phi(\tilde{u})\tilde{u}_j,\nonumber
\end{eqnarray}
which, combined with $ H = \tilde H$, imply 
\begin{equation}\label{3.5}
\sum_iv_{ij,i} = -(n-1)\left(\tilde{\varphi}\Phi(\tilde{u})\tilde{u}_j-\varphi\Phi(u)u_j\right).
\end{equation}
We still let $w$  be some weighted single-variable  function on $I$ which is to be chosen  later. 
We calculate the following integral
\begin{eqnarray}\label{3.8}
&&-\int_{\Sigma}\sum_{i,j}(w(\tilde{u})u_i-w(u)\tilde{u}_i)_jv_{ji}\\
&=&\int_{\Sigma}\sum_{i,j}(w(\tilde{u})u_i-w(u)\tilde{u}_i)v_{ji,j}\nonumber\\
&=&(n-1)\int_{\Sigma}w(u)\tilde{\varphi}\Phi(\tilde{u})|\nabla\tilde{u}|^2+(n-1)\int_{\Sigma}w(\tilde{u})\varphi\Phi(u)|\nabla u|^2\nonumber\\
&&-(n-1)\int_{\Sigma}(w(u)\varphi\Phi(u)+w(\tilde{u})\tilde{\varphi}\Phi(\tilde{u}))\nabla u\cdot\nabla\tilde{u}.\nonumber
\end{eqnarray}
Here, in the second equality,   \eqref{3.5} has been used.

On the other hand, we have
$$(w(\tilde{u})u_i-w(u)\tilde{u}_i)_j=w_u(\tilde{u})\tilde{u}_ju_i-w_u(u)\tilde{u}_iu_j+w(\tilde{u})u_{ij}-w(u)\tilde{u}_{ij}, $$
where, by \eqref{0.1}, 
\begin{eqnarray}
u_{ij}=f(u)\delta_{ij}-h_{ij}\varphi,\ \ \tilde{u}_{ij}=f(\tilde{u})\delta_{ij}-\tilde{h}_{ij}\tilde{\varphi}.\nonumber
\end{eqnarray} 
Therefore, 
\begin{eqnarray}\label{3.10}
&&-\sum_{i,j}(w(\tilde{u})u_i-w(u)\tilde{u}_i)_jv_{ji}\\
&=&-\sum_{i,j}(w_u(\tilde{u})\tilde{u}_ju_i-w_u(u)\tilde{u}_iu_j)v_{ij}\nonumber\\
&&+w(u)\sum_{i,j}(f(\tilde{u})\delta_{ij}-\tilde{h}_{ij}\tilde{\varphi})v_{ij}-w(\tilde{u})\sum_{i,j}(f(u)\delta_{ij}-h_{ij}\varphi )v_{ij}\nonumber\\
&=&(w_u(u)-w_u(\tilde{u}))\sum_{i,j}\tilde{u}_iu_jv_{ij}+w(\tilde{u})\varphi\sum_{i,j}h_{ij}v_{ij}-w(u)\tilde{\varphi}\sum_{i,j}\tilde{h}_{ij}v_{ij}\nonumber\\
&=&-\int^{\tilde{u}}_{u}w_{ss}(s)ds\sum_{i,j}u_i\tilde{u}_jv_{ij}\nonumber\\
&&+w(u)\tilde{\varphi}\left(\frac{(n-1)n}{2}(f_u(\tilde{u})-f_u(u))-(n-2)(n-1)(\tilde{\rho}\Phi(\tilde{u})-\rho\Phi(u))\right.\nonumber\\
&&\ \ \ \ \ \ \ \ \ \ \ \ \  \left.+(n-1)(\varphi^2\Phi(u)-\tilde{\varphi}^2\Phi(\tilde{u}))-\frac{1}{2}\sum_{i,j}v_{ij}^2\right)\nonumber\\
&&-w(\tilde{u})\varphi\left(\frac{(n-1)n}{2}(f_u(\tilde{u})-f_u(u))-(n-2)(n-1)(\tilde{\rho}\Phi(\tilde{u})-\rho\Phi(u))\right.\nonumber\\
&&\ \ \ \ \ \ \ \ \ \ \ \ \  \left.+(n-1)(\varphi^2\Phi(u)-\tilde{\varphi}^2\Phi(\tilde{u}))+\frac{1}{2}\sum_{i,j}v_{ij}^2\right)\nonumber,
\end{eqnarray}
where we have used  \eqref{3.1-g},  \eqref{3.3} in the second, the third equality, respectively. 
By  \eqref{3.8} and  \eqref{3.10}, we thus have
\begin{equation} \label{3.11}
\begin{split}
0 =  & \ (n-1)\int_{\Sigma}\Big[w(u)\tilde{\varphi}\Phi(\tilde{u})(|\nabla \tilde{u}|^2+\tilde{\varphi}^2)+w(\tilde{u})\varphi\Phi(u)(|\nabla u|^2+\varphi^2) \\ 
& \ -(w(u)\varphi\Phi(u)+w(\tilde{u})\tilde{\varphi}\Phi(\tilde{u}))(\nabla u\cdot\nabla \tilde{u}+\varphi\tilde{\varphi})\Big] \\
&+\int_{\Sigma}\sum_{i,j}u_i\tilde{u}_jv_{ij}  \int^{\tilde{u}}_uw_{ss}(s)ds  -\frac{(n-1)n}{2}\int_{\Sigma}w(u)\tilde{\varphi}\int^{\tilde{u}}_uf_{ss}(s)ds \\
&+\frac{(n-1)n}{2}\int_{\Sigma}w(\tilde{u})\varphi\int^{\tilde{u}}_uf_{ss}(s)ds+\frac{1}{2}\int_{\Sigma}(w(u)\tilde{\varphi}+w(\tilde{u})\varphi)\sum_{i,j}v_{ij}^2 \\ 
&+(n-2)(n-1)\int_{\Sigma}(w(u)\tilde{\varphi}-w(\tilde{u})\varphi)\int^{\tilde{u}}_u(\rho\Phi(s))_sds.
\end{split}
\end{equation}
To handle  the first integral  on the right side of \eqref{3.11}, we note that, by \eqref{0.1},  
\begin{equation}\label{3.12-g}
\begin{split}
& \ w(u)\tilde{\varphi}\Phi(\tilde{u})2\tilde{\rho}-w(u)\tilde{\varphi}\Phi(u)2\rho  
 +w(\tilde{u})\varphi\Phi(u)2\rho \\
 & \  -w(\tilde{u})\varphi\Phi(\tilde{u})2\tilde{\rho}
  +w(u)\tilde{\varphi}\Phi(u)2\rho+w(\tilde{u})\varphi\Phi(\tilde{u})2\tilde{\rho}\\
& \ -(w(u)\varphi\Phi(u)+w(\tilde{u})\tilde{\varphi}\Phi(\tilde{u}))(\nabla u\cdot \nabla\tilde{u}+\varphi\tilde{\varphi}) \\
= & \ w(u)\tilde{\varphi}\int^{\tilde{u}}_u(2\Phi(s)\rho(s))_sds-w(\tilde{u})\varphi\int^{\tilde{u}}_u(2\Phi(s)\rho(s))_sds \\
& \ +w(u)\tilde{\varphi}\Phi(u)(|\nabla u|^2+\varphi^2)+w(\tilde{u})\varphi\Phi(\tilde{u})(|\nabla\tilde{u}|^2+\tilde{\varphi}^2) \\
& \ -w(u)\varphi\Phi(u)(\nabla u\cdot\nabla\tilde{u}+\varphi\tilde{\varphi})-w(\tilde{u})\tilde{\varphi}\Phi(\tilde{u})(\nabla u\cdot\nabla\tilde{u}+\varphi\tilde{\varphi}) \\
=& \ w(u)\tilde{\varphi}\int^{\tilde{u}}_u(2\Phi(s)\rho(s))_sds-w(\tilde{u})\varphi\int^{\tilde{u}}_u(2\Phi(s)\rho(s))_sds \\
& \ +\tilde{\varphi}\nabla G(u)\cdot\nabla u+\varphi\nabla G(\tilde{u})\cdot\nabla\tilde{u}-\varphi\nabla G(u)\cdot\nabla \tilde{u}-\tilde{\varphi}\nabla G(\tilde{u})\cdot \nabla u, 
\end{split}
\end{equation}
where  we define $G(s)=\int w(s)\Phi(s)ds.$
Integrating  by parts, we have
\begin{equation}\label{3.13}
\begin{split}
&\int_{\Sigma}\Big[\tilde{\varphi}\nabla G(u)\cdot\nabla u+\varphi\nabla G(\tilde{u})\cdot\nabla\tilde{u}-\varphi\nabla G(u)\cdot\nabla\tilde{u}-\tilde{\varphi}\nabla G(\tilde{u})\cdot\nabla u\Big] \\
=&\int_{\Sigma}\Big[G(u)\nabla(\varphi\nabla\tilde{u})-G(u)\nabla(\tilde{\varphi}\nabla u)+G(\tilde{u})\nabla(\tilde{\varphi}\nabla u)-G(\tilde{u})\nabla(\varphi\nabla\tilde{u})\Big] \\
=&\int_{\Sigma}\Big(G(\tilde{u})-G(u)\Big)\Big(\nabla\tilde{\varphi}\cdot\nabla u-\nabla\varphi\cdot\nabla\tilde{u}+\tilde{\varphi}\Delta u-\varphi\Delta\tilde{u}\Big) \\
=&\int_{\Sigma}\left[\sum_{i,j}(\tilde{h}_{ij}\tilde{u}_iu_j-h_{ij}u_i\tilde{u}_j)+\tilde{\varphi}(nf(u)-H\varphi)-\varphi(nf(\tilde{u})-\tilde{H}\tilde{\varphi})\right] \left(\int^{\tilde{u}}_uw(s)\Phi(s)ds\right)\\
=&\int_{\Sigma} \sum_{i,j}v_{ij}u_i\tilde{u}_j \left(\int^{\tilde{u}}_uw(s)\Phi(s)ds\right)
+\int_{\Sigma}\Big(nf(u)\tilde{\varphi}-nf(\tilde{u})\varphi\Big)\int^{\tilde{u}}_uw(s)\Phi(s)ds , 
\end{split}
\end{equation}
where we have used \eqref{0.1},  \eqref{var} and  \eqref{3.1-g}.

Now, combing  \eqref{3.12-g}, \eqref{3.13} and \eqref{3.11}, we have
\begin{equation}\label{3.14}
\begin{split}
0 =& \ \frac{1}{2}\int_{\Sigma}(w(u)\tilde{\varphi}+w(\tilde{u})\varphi)\sum_{i,j}v_{ij}^2 \\
& \ + \int_{\Sigma} \sum_{i,j}u_i\tilde{u}_jv_{ij} \int^{\tilde{u}}_u(w_{ss}(s)+(n-1)w(s)\Phi(s))ds \\
& \ + \int_{\Sigma} (w(u)\tilde{\varphi}-w(\tilde{u})\varphi){  \int^{\tilde{u}}_u\left[(n-1)n(\Phi(s)\rho(s))_s-\frac{(n-1)n}{2}f_{ss}(s)\right]ds} \\
& \ +(n-1)n\int_{\Sigma} (f(u)\tilde{\varphi}-f(\tilde{u})\varphi) \int^{\tilde{u}}_uw(s)\Phi(s)ds .
\end{split}
\end{equation}
Since  $(\rho\Phi)_{u}=\rho\Phi_u+\rho_u\Phi=\rho\Phi_u+f\Phi,$
by \eqref{2.10} we have 
\begin{equation} \label{eq-rhophi}
2(\rho\Phi)_u-f_{uu}=2\rho\Phi_u+2f\Phi-f_{uu} =-2f\Phi. 
\end{equation}
Inserting \eqref{eq-rhophi} into \eqref{3.14}, we thus obtain 
\begin{equation}\label{3.16}
\begin{split}
0 = &\frac{1}{2}\int_{\Sigma}(w(u)\tilde{\varphi}+w(\tilde{u})\varphi)\sum_{i,j}v_{ij}^2 \\
&+\int_{\Sigma}  \sum_{i,j}u_i\tilde{u}_jv_{ij}  \int^{\tilde{u}}_u(w_{ss}(s)+(n-1)\Phi(s)w(s))ds \\
&+(n-1)n\int_{\Sigma}\tilde{\varphi}\left[-w(u)\int^{\tilde{u}}_uf(s)\Phi(s)ds+f(u)
\int^{\tilde{u}}_uw(s)\Phi(s)ds\right] \\
&+(n-1)n\int_{\Sigma}\varphi\left[w(\tilde{u})\int^{\tilde{u}}_uf(s)\Phi(s)ds-f(\tilde{u})
\int^{\tilde{u}}_uw(s)\Phi(s)ds\right] .
\end{split}
\end{equation}

Now suppose  $(N, d s^2)$ 
satisfies \eqref{eq-super-static}. 
Let  $  (r_0, r_1) \subset I$  be an interval so that the coordinate $r$ varies between 
$r_0$ and $r_1$ for all points on $M$ and $ \tilde M$. 
Similar to Section \ref{sec-inf-rigi},  
we  let  $ w$  be a solution to the ODE
$$w_{uu}+(n-1)\Phi w=0 . $$ 
With such a choice of $ w$, the second integral on the right side of \eqref{3.16} vanishes. 
To handle the last two integrals, we note that 
\begin{equation}\label{4.13}
\begin{split}
& f(u)\int^{\tilde{u}}_uw(s)\Phi(s)ds-w(u)\int^{\tilde{u}}_uf(s)\Phi(s) \, ds\\
=& \int^{\tilde{u}}_uf(u)f(s)\left[\frac{w(s)}{f(s)}-\frac{w(u)}{f(u)}\right]\Phi(s) \, ds ,
\end{split}
\end{equation}
and
\begin{equation}\label{4.13-t}
\begin{split}
& f( \tilde u)\int^{\tilde{u}}_uw(s)\Phi(s)ds-w( \tilde u)\int^{\tilde{u}}_uf(s)\Phi(s) \, ds\\
=& \int^{\tilde{u}}_u f(\tilde u)f(s)\left[\frac{w(s)}{f(s)}-\frac{w ( \tilde u)}{f( \tilde u)}\right]\Phi(s) \, ds . 
\end{split}
\end{equation}
Thus, we want $ w/f$ to have a suitable  monotonic  property depending on the sign of $ \Phi$. 
If  $ \Phi \ge 0 $, we specify   the initial conditions of $ w$ at $ r = r_0$ so that 
 $w(r_0)>0 $, $ w_u(r_0)>0$
and 
\be \label{eq-initial-w-g}
w_u(r_0)f(r_0)-w(r_0)f_u(r_0) \ge 0 .
\ee
Then, by the argument in Section \ref{sec-inf-rigi}, we have  $ w > 0 $ and 
$ w_u f - w f_u   \ge 0 $. The latter then implies 
$ w/f $ is monotonically non-decreasing. 
Hence, it follows from \eqref{3.16}, \eqref{4.13} and \eqref{4.13-t} that 
$ v_{ij} = 0 $, provided $ \varphi > 0 $ and $ \tilde \varphi > 0 $.
If $ \Phi > 0 $, by choosing the inequality in \eqref{eq-initial-w-g}
to be strict, we then have $ w_u f - w f_u  >  0 $. In this case, 
$ w/f $ is strictly  increasing, and 
we conclude from  \eqref{3.16}, \eqref{4.13} and \eqref{4.13-t} that 
$ u = \tilde u $ everywhere  on $ \Sigma$.
 As a result, $M$ and $ \tilde M$ differ by a rotation of $(N, d s^2)$.

The case $ \Phi \le 0 $ (and $ \Phi < 0 $) are proved in a similar way 
by choosing the initial  conditions of $ w$ at $ r = r_1$ so that 
$w(r_1)>0$, $w_u(r_1)<0$ and 
\be \label{eq-initial-w-2-g}
w_u(r_1)f(r_1)-w(r_1)f_u(r_1) \le 0. 
\ee
When $ \Phi < 0 $, we choose the inequality in \eqref{eq-initial-w-2-g}  to be strict.

Therefore, the discussion above has  proved the following theorem, which is 
a restatement of Theorem \ref{theo-intro}.

\begin{theo}
Suppose $(N, ds^2) = \left( I \times \mathbb{S}^n, \frac{1}{f^2(r)}dr^2+r^2d\sigma \right)$
satisfies \eqref{eq-super-static},
with  either $ \Phi \ge 0 $ or $ \Phi \le 0 $.
Let $g$ be a Riemannian metric on  $ \Sigma = \mathbb{S}^n$.
Suppose  $(\Sigma, g)$ can be isometrically immersed into $(N,ds^2)$ as two hypersurfaces
 $M$ and $ \tilde M$, both of which have positive support function.
If $M$ and $ \tilde M$ have the same mean curvature, 
then $M$ and $\tilde{M}$ have the same second fundamental form. 
If in addition $\Phi>0$ or $ \Phi < 0$, then $M$ and $\tilde{M}$
only differ by  a rotation of the ambient space $(N, d s^2)$.    
\end{theo}

\begin{remark} 
In view of the above proof, 
to draw  the conclusion $ u = \tilde u $, 
it suffices to require the set  $ \{ \Phi > 0 \} $ 
is dense in $ I$.
On the other hand, if 
the set $\{ \Phi =0 \} $ contains some open interval $\tilde I $,
then it is easily checked  $(N,d s^2)$ contains a ring $ R = \tilde I  \times \mathbb{S}^n$ which is  part of a space form. In this case, if $M$ is contained in $R$, we can ``translate" $M$ in $ R$  by an isometry of that space form.  Obviously, such a ``translation" preserves  the second fundamental form of $M$, but it may not be the restriction of a global  isometry  of  $(N , ds^2)$.
\end{remark}

\section{Other related rigidities} \label{sec-other}

In this last section, we will describe several related rigidity results. 

First, we consider rigidity of isometric hypersurfaces with the same $\sigma_2$-curvature.
Recall that the $\sigma_2$-curvature of  a hypersurface $M$ is defined by 
\be \label{eq-sigma-2}
 \sigma_2 (h) = \sum_{ i < j} \kappa_i \kappa_j 
\ee
where $ \{ \kappa_i \}$ are the principal curvature of $M$.

\begin{theo} \label{theo-inf-sigma2}
Let $(N, ds^2) = \left( I \times \mathbb{S}^n, \frac{1}{f^2(r)}dr^2+r^2d\sigma \right)$
be a warped product space  with $ \Phi \neq 0 $.
Let $g$ be a Riemannian metric on  
$ \Sigma =  \mathbb{S}^n$.
 If $(\Sigma, g)$ can be isometrically  immersed into $(N, d s^2)$ as a hypersurface $M$ with 
 nowhere vanishing mean curvature and nowhere vanishing support function, 
 then the infinitesimal rigidity of $M$ with fixed $ \sigma_2$-curvature holds. 
 Precisely, this means that 
  if $(\Sigma, g)$ admits a 
family of isometric immersion $\{ \iota_t \}$ into $(N, d s^2)$  such that $ M_t = \iota_t (\Sigma)$ 
has the same $\sigma_2$-curvature and $M_0 = M$, 
then the  linearization of the second fundamental form of $M_t$ at $M$ is trivial.       
\end{theo}
\begin{proof}
We use the same notations in Section \ref{sec-inf-rigi}. 
By the assumptions, 
\begin{equation}\label{5.1}
\dot{\sigma}_2 (h) =\left.\frac{d\sigma_2(h)}{dt}\right|_{t=0}=0; \
\dot{g}_{ij}=\left.\frac{d g_{ij}}{dt}\right|_{t=0}=0.
\end{equation}
Hence, the linearization of \eqref{2.5} gives 
\begin{eqnarray}\label{5.2}
\frac{n}{2}f_{uu}\dot{u}-(n-2)(f\Phi+\rho\Phi_u)\dot{u}-2\varphi\dot{\varphi}\Phi-\varphi^2\Phi_{u}\dot{u}=0.
\end{eqnarray}
Multiplying $\dot{u}$ in  both sides of \eqref{5.2}  and  integrating on $(\Sigma, g)$, we have 
\begin{equation}
\begin{split}
0 =&\frac{n}{2}\int_\Sigma f_{uu}\dot{u}^2-2\int_\Sigma (f\dot{u}-\nabla u\cdot\nabla\dot{u})\Phi\dot{u}-\int_\Sigma (2\rho-|\nabla u|^2)\Phi_{u}\dot{u}^2 \\
&-(n-2)\int_\Sigma (f\Phi+\rho\Phi_u)\dot{u}^2 \\
=&\frac{n}{2}\int_\Sigma (f_{uu}-2f\Phi-2\rho\Phi_{u})\dot{u}^2+\int_\Sigma \Phi \nabla u\cdot\nabla \dot{u}^2+\int_\Sigma |\nabla u|^2\Phi_{u}\dot{u}^2 \\
=&\frac{n}{2}\int_\Sigma (f_{uu}-2f\Phi-2\rho\Phi_{u})\dot{u}^2-\int_\Sigma \Phi (\Delta u) \dot{u}^2 \\
=&\frac{n}{2}\int_\Sigma (f_{uu}-2f\Phi-2\rho\Phi_{u})\dot{u}^2+\int_\Sigma \Phi (H\varphi-nf)\dot{u}^2 \\
=&\frac{n}{2}\int_\Sigma (f_{uu}-4f\Phi-2\rho\Phi_{u})\dot{u}^2+\int_\Sigma \Phi H\varphi\dot{u}^2 .
\end{split}
\end{equation}
 Thus, by \eqref{2.10}, we have $\dot{u}=0$ 
since $ \Phi H \varphi \neq 0 $.  In view of \eqref{1.3} and \eqref{1.4},
we conclude  $\dot{h}_{ij}=0$. 
\end{proof}

\begin{remark}
Unlike the fixed mean curvature problem, 
Theorem \ref{theo-inf-sigma2} does not need 
the assumption \eqref{eq-super-static}, 
and the Codazzi equations are not used in the  proof. 
\end{remark}

By imposing a stronger assumption on $\Phi$,  
we can prove the global rigidity result stated in Theorem \ref{theo-intro-sigma-2}.

\begin{theo}
Let $(N, ds^2) = \left( I \times \mathbb{S}^n, \frac{1}{f^2(r)}dr^2+r^2d\sigma \right)$
be a warped product space  with $\Phi\Phi_u>0$. 
Let $g$ be a Riemannian metric on   $ \Sigma =  \mathbb{S}^n$. 
Suppose  $(\Sigma, g)$ can be isometrically  immersed into $(N, d s^2)$ as 
two hypersurfaces $M$ and $\tilde M$. If $M$ and $ \tilde M$ have the same $\sigma_2$-curvature, 
then they differ by a rotation of the ambient space $(N, d s^2)$.
 \end{theo}
 \begin{proof}
 We make  use of  the maximum principle. 
 To illustrate the idea, we first give another proof of the infinitesimal rigidity in this setting.
Using  \eqref{0.1}, \eqref{2.10} and \eqref{1.3},   we can rewrite \eqref{5.2} as 
 \be
 nf\Phi\dot{u}+2\nabla u\cdot\nabla\dot{u}\Phi+|\nabla u|^2\Phi_u\dot{u}=0.
 \ee
  At the maximum and minimum  points of the function $\dot{u}$, we have  
 $\nabla \dot{u}=0$, which implies $\dot{u}=0$ at these  points by the assumption $ \Phi \Phi_u > 0$. 
 Thus, we have $\dot{u}=0$ everywhere on $\Sigma$. 
 
 To prove the global rigidity, we use the same notations from Section \ref{sec-global-rigi}.
By \eqref{3.2}, we have 
 $$\frac{n}{2}f_u-(n-2)\rho\Phi-\varphi^2\Phi=\frac{n}{2}f_u(\tilde{u})-(n-2)\tilde{\rho}\Phi(\tilde{u})-\tilde{\varphi}^2\Phi(\tilde{u}).$$ 
We can rewrite it as 
 \begin{equation}
 \begin{split}
  &\frac{n}{2}\int^{\tilde{u}}_uf_{ss}(s)ds \\
  = & \ n\int_{u}^{\tilde{u}}(\rho\Phi)_s ds+ |\nabla u|^2\Phi(u)-|\nabla\tilde{u}|^2\Phi(\tilde{u}) \\
  = & \ n\int^{\tilde{u}}_u\left(f\Phi+\rho\Phi_s\right)ds+\nabla(u-\tilde{u})\cdot\nabla(u+\tilde{u})\Phi(u)+|\nabla\tilde{u}|^2(\Phi(u)-\Phi(\tilde{u})) ,
  \end{split}
  \end{equation}
  which together with \eqref{2.10}  implies
  \be  \label{eq-u-ud} 
  n\int^{\tilde{u}}_uf\Phi d s = \nabla(u-\tilde{u})\cdot\nabla(u+\tilde{u})\Phi(u) 
  -|\nabla\tilde{u}|^2\int^{\tilde{u}}_u\Phi_sds.
  \ee
  At the maximum and minimum  points of the function $u-\tilde{u}$, we have
   $\nabla(u-\tilde{u})=0$. Thus, \eqref{eq-u-ud} implies   $u=\tilde{u}$ at these  points 
   since $ \Phi \Phi_u > 0 $.
   Therefore,  $u=\tilde{u}$ on the entire $\Sigma$, which implies 
   $M$ and $ \tilde M$ differ by a rotation of $(N,d s^2)$.
 \end{proof}

The techniques used in Sections \ref{sec-inf-rigi} and \ref{sec-global-rigi} 
indeed relate to a revisit of 
the rigidity of isometric embeddings of surfaces   into space forms. 
We give a brief review of the history of this rigidity problem. 
It is known that the infinitesimal rigidity of closed convex surfaces in Euclidean spaces was shown by 
Cohn-Vossen \cite{CV2} and was simplified by Blaschke  \cite{B1} using Minkowski identities.  The infinitesimal rigidity in hyperbolic spaces was discussed by Lin-Wang \cite{LW}. In \cite{Li-Wang}, the first and the third authors gave an alternative proof of the infinitesimal rigidity for convex surfaces in space forms. The global rigidity was obtained by Cohn-Vossen \cite{CV1} for convex surfaces  in Euclidean spaces. In space forms, these rigidities  are also known valid (cf. \cite{D, GS}).

Below we give an alternate proof of the global rigidity using the methods from Sections \ref{sec-inf-rigi} and \ref{sec-global-rigi}.

\begin{theo} \label{theo-rigidity-space-form}
Suppose $(N, ds^2)$ is a $3$-dimensional  space form.
Let $ g $ be a metric on the $2$-sphere $ \mathbb{S}^2$.
If $(\mathbb{S}^2, g)$ can be isometrically  embedded into $(N, d s^2)$ as a  strictly convex surface $\Sigma$, then the infinitesimal rigidity and global rigidity of $\Sigma$ hold.
\end{theo}
\begin{proof}
In a space form with constant sectional curvature  $ - k$, we have 
$$\Phi=0, f_u=k. $$ Thus, the Gauss equation becomes 
$$\frac{\det h}{\det g}=K+k.$$ 
By ``translating" $\Sigma$ in $(N, ds^2)$, we can always assume that 
the support function  $ \varphi$ of $\Sigma$ is  positive since $ \Sigma$ is strictly convex.     

First we consider the infinitesimal rigidity. The linearization of the Gauss-Codazzi equations gives 
\begin{eqnarray}\label{6.1}
&h_{11}\dot{h}_{22}+h_{22}\dot{h}_{11}-2h_{12}\dot{h}_{12}=0,\\
&\dot{h}_{ij,k}=\dot{h}_{ik,j}.
\end{eqnarray}
By \eqref{1.4}, we have
$$\dot{u}_{ij}=k\dot{u}\delta_{ij}-\dot{h}_{ij}\varphi-h_{ij}\dot{\varphi}.$$ 
Integrating by parts, we have
\begin{eqnarray}\label{6.3}
\ \ 0=\int_{\Sigma} \Big[(f(u)\dot{u}_1)_1\dot{h}_{22}+(f(u)\dot{u}_2)_2\dot{h}_{11}-(f(u)\dot{u}_1)_2\dot{h}_{21}-(f(u)\dot{u}_2)_1\dot{h}_{12}\Big].
\end{eqnarray}
On the other hand, we have 
$$(f(u)\dot{u}_i)_j=f_uu_j\dot{u}_i+f(u)\dot{u}_{ij}=ku_j\dot{u}_i+f(u)(k\dot{u}\delta_{ij}-\dot{h}_{ij}\varphi-h_{ij}\dot{\varphi}).$$
Inserting the above equality into \eqref{6.3}, we have 
\begin{eqnarray}\label{6.4}
\\
&&\int_{\Sigma}\Big[ku_1\dot{u}_1\dot{h}_{22}+kf(u)\dot{u}\dot{h}_{22}+ku_2\dot{u}_2\dot{h}_{11}+kf(u)\dot{u}\dot{h}_{11}-ku_1\dot{u}_2\dot{h}_{12}-ku_2\dot{u}_1\dot{h}_{21}\Big]\nonumber\\
&&-2\int_{\Sigma}(\dot{h}_{11}\dot{h}_{22}-\dot{h}_{12}\dot{h}_{21})\varphi=0\nonumber.
\end{eqnarray}
Here we have used \eqref{6.1}.  Integrating  by part, we also have 
\begin{eqnarray}\label{6.5}
&&\int_{\Sigma}\Big[ku_1\dot{u}_1\dot{h}_{22}+ku_2\dot{u}_2\dot{h}_{11}-ku_1\dot{u}_2\dot{h}_{12}-ku_2\dot{u}_1\dot{h}_{21}\Big]\\
&=&-k\int_{\Sigma}\dot{u}\Big[({u}_1\dot{h}_{22})_1+({u}_2\dot{h}_{11})_2-({u}_2\dot{h}_{12})_1-({u}_1\dot{h}_{21})_2\Big]\nonumber\\
&=&-k \int_{\Sigma}\dot{u}\Big[u_1(\dot{h}_{221}-\dot{h}_{122})+u_2(\dot{h}_{112}-\dot{h}_{211})+u_{11}\dot{h}_{22}+u_{22}\dot{h}_{11}-2u_{12}\dot{h}_{12}\Big]\nonumber\\
&=&-k \int_{\Sigma}\dot{u}\Big[\dot{h}_{22}(f(u)-h_{11}\varphi)+\dot{h}_{11}(f(u)-h_{22}\varphi)-2\dot{h}_{12}(-h_{12}\varphi)\Big] \nonumber\\
&=&-k\int_{\Sigma}\dot{u}(\dot{h}_{11}+\dot{h}_{22})f(u).\nonumber
\end{eqnarray}
Combing \eqref{6.4} and \eqref{6.5}, we obtain 
\begin{eqnarray}\label{6.6}
\int_{\Sigma}(\dot{h}_{11}\dot{h}_{22}-\dot{h}^2_{12})\varphi=0.
\end{eqnarray}
At any  given point,  by computing in a normal coordinate chart that 
diagonalizes $ h_{ij}$, it is easily seen that  \eqref{6.1}, together with the convexity of $\Sigma$,  implies 
$$\dot{h}_{11}\dot{h}_{22}-\dot{h}_{12}^2\leq 0.$$  
Hence, from   \eqref{6.6}, we obtain $\dot{h}=0$.

Next, we modify the previous argument to prove the global rigidity.
Suppose $\Sigma$ and $\tilde{\Sigma}$ are two  isometrically  embedded convex surfaces. 
Suppose their second fundamental forms are $h$ and $\tilde{h}$, respectively. 
Then, $v=\tilde{h}-h$ satisfies 
\begin{eqnarray}
&&h_{11}v_{22}+h_{22}v_{11}-2h_{12}v_{12}+v_{11}v_{22}-v_{12}^2=0,\nonumber\\
&&\tilde{h}_{11}v_{22}+\tilde{h}_{22}v_{11}-2\tilde{h}_{12}v_{12}-v_{11}v_{22}+v_{12}^2=0\nonumber, 
\end{eqnarray}
by the Gauss equations. 
Thus, we have 
\begin{eqnarray}\label{6.1-g}
(h_{11}+\tilde{h}_{11})v_{22}+(h_{22}+\tilde{h}_{22})v_{11}-2(h_{12}+\tilde{h}_{12})v_{12}=0.
\end{eqnarray}
If we  compute  in a normal coordinate chart that 
diagonalizes  $ h + \tilde{h} $, then \eqref{6.1-g} and
 the convexity of $ \Sigma$ and $ \tilde \Sigma$ show
 $v_{11}$ and $v_{22}$ have different sign. Thus, we have 
\begin{eqnarray}\label{6.8}
v_{11}v_{22}-v_{12}^2\leq 0.
\end{eqnarray}
  The Codazzi equations imply 
\begin{eqnarray}
v_{ij,k}=v_{ik,j}.
\end{eqnarray} 
Integrating  by part, we have 
\begin{eqnarray}\label{6.10}
 0&=&\int_{\Sigma} \Big[(f(\tilde{u})u_1-f(u)\tilde{u}_1)_1v_{22}+(f(\tilde{u})u_2-f(u)\tilde{u}_2)_2v_{11}\\
&& -(f(\tilde{u})u_1-f(u)\tilde{u}_1)_2v_{21}-(f(\tilde{u})u_2-f(u)\tilde{u}_2)_1v_{12}\Big].\nonumber
\end{eqnarray}
On the other hand, we have 
\begin{eqnarray}
&&(f(\tilde{u})u_i-f(u)\tilde{u}_i)_j\nonumber\\
&=&f(\tilde{u})u_{ij}-f(u)\tilde{u}_{ij} +k(\tilde{u}_{j}u_i-u_j\tilde{u}_i)\nonumber\\
&=&f(\tilde{u})(f(u)\delta_{ij}-h_{ij}\varphi)-f(u)(f(\tilde{u})\delta_{ij}-\tilde{h}_{ij}\tilde{\varphi})+k(\tilde{u}_{j}u_i-u_j\tilde{u}_i)\nonumber\\
&=&-f(\tilde{u})h_{ij}\varphi+f(u)\tilde{h}_{ij}\tilde{\varphi}+k(\tilde{u}_{j}u_i-u_j\tilde{u}_i)\nonumber.
\end{eqnarray}
Inserting the above equality into \eqref{6.10}, we obtain 
\begin{eqnarray}\label{6.11}
\\
0&=&\int_{\Sigma}\Big[-f(\tilde{u})(h_{11}v_{22}+h_{22}v_{11}-2h_{12}v_{12})\varphi+f(u)(\tilde{h}_{11}v_{22}+\tilde{h}_{22}v_{11}-2\tilde{h}_{21}v_{12})\tilde{\varphi}\Big]\nonumber\\
&=&\int_{\Sigma}(f(\tilde{u})\varphi+f(u)\tilde{\varphi})(v_{11}v_{22}-v_{12}^2)\nonumber.
\end{eqnarray}
Here we have used \eqref{6.1-g}.  
Combing \eqref{6.8} and \eqref{6.11}, we conclude  $v=0$.  
Thus, $ \Sigma$ and $ \tilde \Sigma$ have the same first and second fundamental forms. 
Since $(N, d s^2) $  is a space form, we conclude that 
$ \Sigma$ and $ \tilde \Sigma$ are same up to an isometry of $(N, d s^2)$.
\end{proof}

\vspace{.3cm}

\noindent {\bf Acknowledgement}. The main part of this work was completed when the first author was visiting Fudan University in 2017. He would like to thank Fudan University for the hospitality and support.
The authors also thank Siyuan Lu for helpful conversations.


\begin{thebibliography}{99}

\bibitem{Alex}  A.D. Alexandrov, 
{\em Uniqueness theorems for surfaces in the large I}, Vesnik Leningrad Univ. 11, 5-17 (1956).

\bibitem{Bartnik} 
R. Bartnik, {\em New definition of quasilocal mass}, Phys. Rev. Lett., {\bf 62} (1989),  2346--2348.

\bibitem{B1} W. Blaschke, 
{\em Vorlesungen Uber Differentialgeometrie}, Vol. 1, Julis Springer, Berlin, 1924.

\bibitem{Bray} 
H. L. Bray,  {\em Proof of the Riemannian Penrose inequality using the positive mass theorem}.
J. Differential Geom. \textbf{59} (2001), no. 2, 177--267. 

\bibitem{B2} S. Brendle, {\em Constant mean curvature surfaces in warped product space}, Publ. Math. Inst. Hautes etudes Sci.(117) 2013, 247-269.

\bibitem{BY1} 
J. D. Brown and  J. W. York, Jr., {\em Quasilocal energy in general relativity}, 
in {Mathematical aspects of classical field theory (Seattle, WA,   1991)}, 
volume 132 of {Contemp. Math.}, pages 129--142. Amer. Math. Soc., Providence, RI, 1992.

\bibitem{BY2} 
J. D. Brown and  J. W. York, Jr., 
{\em  Quasilocal energy and conserved charges derived from the   gravitational action},
{Phys. Rev. D (3)} \textbf{47} (1993), no. 4, 1407--1419.

\bibitem{Car} \'E. Cartan, {\em Sur les couples de surfaces applicable avec conservation des coupures principles}, Bull. Sci. Math. 66 (1942) 1-30.  

\bibitem{CX} J. Chang and L. Xiao, {\em The Weyl problem with nonnegative Gauss curvature in hyperbolic space, }  Canad. J. Math. 67(2015), 107-131. 

\bibitem{CWWY}
P.-N. Chen, M.-T. Wang, Y.-K. Wang and S.-T. Yau, 
{\em Quasi-local energy with respect to a static spacetime},
arXiv:1604.02983.

\bibitem{Chen-Zhang} 
P.-N. Chen and  X. Zhang, {\em A rigidity theorem for surfaces in Schwarzschild manifold}, 
arXiv:1802.00887.

\bibitem{Chern} S.S. Chern, {\em Deformation of surfaces preserving principal curvature},  Differential geometry and complex analysis, H.E. Rauch Menerial Volume (Springer 1985) 155-163. 

\bibitem{CK} A. Gervasio Colares and K. Kenmotsu, {\em Isometric deformation of surfaces in $\mathbb{R}^3$ preserving the mean curvature function}, Pacific  J. Math.  136 (1989) 71-80.  

\bibitem{CV1}  E. Cohn-Vossen, {\em Zwei Satze uber die Starrheit der Eiflachen, Nach. Gesell-schaft Wiss}, Gottingen, Math. Phys.
KL, 1927: 125.134.
\bibitem{CV2} E. Cohn-Vossen, {\em Unstarre geochlossene Flachen}, Math. Ann., 1930, 102: 10.29.

\bibitem{Corvino} 
J. Corvino, 
{\em Scalar curvature deformation and a gluing construction for the Einstein constraint equations}, 
Comm. Math. Phys. \textbf{214} (2000), 137--189.

\bibitem{D} M. Dajczer, {\em Submanifolds and Isometric immersions}, Mathematical Lecture Series 13, 1990.

\bibitem{GLW}
P. Guan, J. Li and M.-T. Wang,
{\em A volume preserving flow and the isoperimetric problem in warped product spaces},
arXiv:1609.08238.

\bibitem{Guan-Lu} P. Guan and S. Lu, {\em Curvature estimates for immersed hypersurfaces in 
Riemannian manifolds},  Invent. math. {\bf 208} (2017), no. 1, 191-215.

\bibitem{GuWZ} P. Guan, Z. Wang, and X. Zhang {\em A proof of the Alexanderov's uniqueness theorem for convex surface in $\mathbb{R}^3$}, Ann. Inst. H. Poincare Anal. Non Lineaire,  (2016) 329-336.

\bibitem{GS} P. Guan and X. Shen, {\em A rigidity theorem for hypersurfaces in higher dimensional space forms,}
Contemporary Mathematics, AMS. V.644, 2015. pp. 61-65.

\bibitem{HI01} 
G. Huisken and T. Ilmanen,
{\em The inverse mean curvature flow and the {R}iemannian {P}enrose inequality},
J. Differential  Geom. \textbf{59} (2001), no. 3, 353--437.

\bibitem{Ken} K. Kenmotsu, {\em An intrinsic characterization of $H$-deformable surfaces}, J. London Math. Soc. 49 (1994) 555-568. 

\bibitem{Li-Wang}  C. Li and Z. Wang, {\em The Weyl problem in warped product spaces}, arXiv:1603.01350.

\bibitem{Lu} S. Lu, {\em On Weyl's embedding problem in Riemannian manifolds}, arXiv:1608.07539.

\bibitem{Lu-Miao} S. Lu and P. Miao, {\em Minimal hypersurfaces and boundary behavior of compact 
manifolds with nonnegative scalar curvature}, arXiv: 1703.08164. 

\bibitem{La} H. B. Lawson Jr., {\em Complete minimal surfaces in $\mathbb{S}^3$,} Ann. of Math. 92 (1970) 335-374.

\bibitem{La-Tr} H. B. Lawson Jr. and R. de Azevedo Tribuzy, {\em On the mean curvature function for compact surfaces}, J. Diff. Geo. (1981) 179-183. 

\bibitem{Lewy} H. Lewy, {\em On the existence of a closed convex sucface realizing a given Riemannian metric},
Proceedings of the National Academy of Sciences, U.S.A., Volume 24, No. 2, (1938), 104-106.

\bibitem{LW} C.-Y. Lin and Y.-K. Wang, {\em On isometric embeddings into anti-de sitter spacetimes}, Int. Math. Res. Notices, No. 16 (2015), 7130-7161.

\bibitem{Mon} S. Montiel, {\em Unicity of constant mean curvature hypersurfaces in some Riemannian manifolds}, Indiana Univ. Math. J. 48, 711-748 (1999)
\bibitem{Nirenberg}
L. Nirenberg, {\em The Weyl and Minkowski problem in differential geometry in the large},
Commun. Pure Appl. Math. \textbf{6} (1953), no. 3, 337--394.

\bibitem{Pogorelov}
A. V. Pogorelov, {\em Regularity of a convex surface with given Gaussian curvature}, 
(Russian) Mat. Sbornik N.S. \textbf{31}(73), (1952), no. 1, 88--103.

\bibitem{Pogorelov2} A.V. Pogorelov,  {\em Some results on surface theory in the large},  Adv.  Math. 1 1964,
fasc. 2, 191-264.

\bibitem{PC} A.V. Pogorelov,  {\em Extrinsic Geometry of Convex Surfaces}, Translations of mathematical mongraphs, Vol. 35, AMS, 1973.

\bibitem{Sch} W. Scherrer, {\em Die Grundgleichungen der Fl\"achentheorie II}, Comm. Math. Helv., 32 (1957), 73-84.

\bibitem{Svec} A. Svec, {\em Determination of a surface by its mean curvature}, Casopis pro pestovani matematiky, 103 (1978), 175-180. 

\bibitem{Tr} R. Tribuzy, {\em A characterization of tori with constant mean curvature in a space form, } Bol. Soc. Brasil. Mat. 11 (1980) 259-274. 

\bibitem{Weyl} H. Weyl, {\em Uber die Bestimmung einer geschlossenen konvexen Flache durch ihr Linienelement},
Vierteljahrsschrift der naturforschenden Gesellschaft, Zurich, 61, (1916), 40-72.

\bibitem{WY1} M.-T. Wang  and  S.-T. Yau,  
{\em Quasilocal mass in general relativity},  Phys. Rev. Lett. 102 (2009), no. 2, no. 021101, 4 pp.

\bibitem{WY2}  M.-T. Wang  and  S.-T. Yau, 
{\em Isometric embeddings into the Minkowski space and new quasi-local mass}, 
Comm. Math. Phys. \textbf{288}  (2009), no. 3, 919-942. 

\bibitem{SY79} 
R. Schoen and S.-T.  Yau,
{\em On the proof of the positive mass conjecture in general relativity},
Commun.  Math. Phys. \textbf{65} (1979), no. 1, 45--76.

\bibitem{Um} M. Umehara, {\em A characterization of compact surfaces with constant mean curvature}, Proc. Amer. Math. Soc., 108 (1990) 483-489.

\end{thebibliography}
\end{document}